%% file: just_cont.tex
\documentclass[12pt,]{article}
\usepackage[left=1in,right=1in,top=1in,bottom=1in,footskip=.4in]{geometry}

\usepackage{amsmath}
\usepackage{amssymb}
\usepackage{stackrel}
\usepackage{graphicx}
\usepackage{epstopdf}
\usepackage{bbm}
\usepackage{geometry}
\usepackage{enumerate}
\usepackage{etoolbox}
\usepackage[ruled,lined,linesnumbered,vlined]{algorithm2e}
\usepackage{tikz}
\usepackage{setspace}

\usepackage{ThmsAndLemmas}
\usepackage{defs}
\usepackage{prjohns2MacrosTA}
	\usepackage{jemacros}

\usepackage{color}

\usepackage{url}

\newcommand{\bw}{{\bf w}}

\newcommand{\longversion}[1]{#1}
\newcommand{\shortversion}[1]{}
\newcommand{\myqed}{}

\title{Projective Splitting with Forward Steps only Requires Continuity}

\author{Patrick R. Johnstone\thanks{Department of Management Science 
                                    and Information Systems, Rutgers Business
                                    School Newark and New Brunswick, Rutgers University}
        \and 
        Jonathan Eckstein$^*$}

\begin{document}
	\allowdisplaybreaks
\maketitle
\input{abstract}
\input{intro}
\input{NotationMathPrelims}

\input{AlgorithmOverview}
\input{AsyncAlg-shortVersion}

\input{lemmas4weakconvGmod}

\input{proofMain}

\input{proofWithUniform}

\bibliographystyle{spmpsci}
\bibliography{refs}

\end{document}

%% file: abstract.tex
\begin{abstract}

A recent innovation in projective splitting algorithms for monotone operator inclusions
has been the development of a procedure using two forward steps instead of the customary
proximal steps for operators that are Lipschitz continuous.
This paper shows that the Lipschitz assumption is 
unnecessary when the forward steps are performed in finite-dimensional spaces: a 
backtracking linesearch yields a convergent algorithm for operators
that are merely continuous with full domain.
\end{abstract}

%% file: intro.tex
\section{Introduction}
For a collection of real Hilbert spaces $\{\calH_i\}_{i=0}^{n}$, consider the problem of finding $z\in\calH_{0}$
such that 
\begin{eqnarray}\label{ProbMono}
\longversion{0\in \sum_{i=1}^{n} G_i^* T_i(G_i z),}
\shortversion{\textstyle{0\in \sum_{i=1}^{n} G_i^* T_i(G_i z),}}
\end{eqnarray}
where $G_i:\calH_{0}\to\calH_i$ are linear and bounded operators, $T_i:\calH_i
\to 2^{\calH_i}$ are maximal monotone operators.  We suppose that 
$T_i$ is continuous with $\dom(T_i)=\calH_i$ for each $i$ in some subset
$\Iforw\subseteq\{1,\ldots,n\}$. A key special case of~\eqref{ProbMono} is
\begin{eqnarray}\label{ProbOpt}
\longversion{\min_{x\in\mathcal{H}_{0}} \sum_{i=1}^{n}f_i(G_i x),}
\shortversion{\textstyle{\min_{x\in\mathcal{H}_{0}} \sum_{i=1}^{n}f_i(G_i x),}}
\end{eqnarray}
where every $f_i:\calH_i\to\mathbb{R}$ is closed, proper and convex, with some
subset of the functions also being Fr\'echet differentiable everywhere. Under
appropriate constraint qualifications, \eqref{ProbMono} and
\eqref{ProbOpt} are equivalent.  
\longversion{Problem~\eqref{ProbOpt} arises in a host of
applications such as machine learning, signal and image processing, inverse
problems, and computer vision; 
see~\cite{boyd2011distributed,combettes2011proximal,combettes2005signal}
for some examples.}

A relatively recently proposed class of operator splitting algorithms which
can solve \eqref{ProbMono}  \longversion{is \emph{projective splitting}.  
It originated
with~\cite{eckstein2008family} and was then generalized to more than two
operators in~\cite{eckstein2009general}.  The related algorithm
in~\cite{alotaibi2014solving} introduced a technique for handling compositions
of linear and monotone operators, and~\cite{combettes2016async} proposed an
extension to ``block-iterative'' and asynchronous operation ---
block-iterative operation meaning that only a subset of the operators making
up the problem need to be considered at each iteration (this approach may be
called ``incremental" in the optimization literature).  A restricted and
simplified version of this framework 
appears in~\cite{eckstein2017simplified}.} \shortversion{is
\emph{projective splitting}~\cite{eckstein2008family,eckstein2009general,alotaibi2014solving,combettes2016async}.} \longversion{Our recent work 
in~\cite{johnstone2018projective} incorporated forward steps
into the projective splitting framework, in place of the customary proximal
(backward) steps, for any Lipschitz continuous operators, and introduced
backtracking and adaptive stepsize rules; see also~\cite{tran2015new}.
The even more recent work
\cite{johnstone2018convergence} derived convergence rates for the method under
various conditions. 

In general, projective splitting offers unprecedented levels of flexibility
compared with previous operator splitting algorithms such
as~\cite{mercier1979lectures,lions1979splitting,davis2015three,tseng2000modified}.
The framework can be applied to arbitary sums of maximal monotone operators,
the stepsizes can vary by operator and by iteration, compositions with linear
operators can be handled, and block-iterative asynchronous implementations
have been demonstrated. Furthermore the number of times each operator is
processed does not need to be equal (either exactly or approximately).

} In~\cite{johnstone2018projective}, we showed that it is possible for
projective splitting to process Lipschitz-continuous operators using a pair of
forward steps rather than the customary proximal step.  In general, the
stepsize must be bounded by the inverse of the Lipschitz constant, but a
backtracking linesearch procedure is available when this constant is unknown.
See also~\cite{tran2015new} for a similar approach to using forward steps in a
more restrictive projective splitting context, without backtracking.


The purpose of this work is to show that this Lipschitz assumption is
unnecessary. It demonstrates that, when the Hilbert spaces $\mathcal{H}_i$ in
\eqref{ProbMono} are finite dimensional for $i\in\Iforw$, the two-forward-step
procedure with backtracking linesearch yields weak convergence to a solution
assuming only simple continuity and full domain of the operators
$T_i$.\footnote{We still speak of weak convergence because the spaces $\mathcal{H}_i$ may be infinite dimensional for $i \not\in \Iforw$. 
If $\mathcal{H}_i$ is infinite dimensional for $i\in\Iforw$,
we can instead require $T_i$ to be Cauchy continuous for all bounded
sequences.} A new argument is required beyond those in
\cite{johnstone2018projective} since the stepsizes resulting from the
backtracking linesearch are no longer guaranteed to be bounded away from $0$.

Theoretically, this result aligns projective splitting with two related
monotone operator splitting methods which utilize two forward steps per iteration and
only require continuity in finite dimension. These are Tseng's forward-backward-forward
method~\cite{tseng2000modified} and the extragradient
method~\cite{korpelevich1977extragradient,iusem1997variant,bello2015variant}.
These methods apply to special cases of Problem \eqref{ProbMono} with $n=2$,
$\Iforw=\{1\}$, and $G_1=G_2=I$; the extragradient method also restricts $T_2$ to
be the normal cone map of a closed convex set.  While the original
extragradient method~\cite{korpelevich1977extragradient} was applied to
variational inequalities under Lipschitz continuity, it was extended in
\cite{iusem1997variant} to include a backtracking linesearch that works under
continuity alone and in \cite{bello2015variant} to solve monotone
inclusions. In fact, the algorithm of
\cite{iusem1997variant,bello2015variant} for just one operator is
almost the special case of projective splitting applied to Problem
\eqref{ProbMono} with one operator and $\Iforw=\{1\}$, the only difference being
a stricter criterion to terminate the backtracking linesearch. Tseng's method
can also be connected to projective splitting with one operator and $\Iback = \{1\}$ following the arguments of~\cite[Section 5.1]{solodov1999hybrid}.

All of these methods can be viewed in contrast with the classical
forward-backward splitting algorithm~\cite{mercier1979lectures}. This method
utlizes a single forward step at each iteration but requires a cocoercivity
assumption which is in general stricter than Lipschitz continuity. Also
disadvantageous is that the choice of stepsize depends on knowledge of the
cocoercivity constant and no backtracking linesearch is known to be available.
Progress was made in a very recent paper \cite{tam2018forward} which modified
the forward-backward method so that it can be applied to (locally) Lipschitz
continuous operators with backtracking for unknown Lipschitz constant. The
locally Lipschitz continuous assumption is stronger than the mere continuity
assumption considered here and in
\cite{tseng2000modified,iusem1997variant,bello2015variant}, and for known
Lipschitz constant the stepsize constraint is more restrictive. 

\longversion{The rest of the paper is organized as follows:
In Section \ref{secNot}, we present notation and some basic background
results. In Section \ref{secAlgOverview}, we precisely state the projective
splitting algorithm and our assumptions, and
collect some necessary results from \cite{johnstone2018projective}. Finally section
\ref{secMain} proves the main result. 
}

%% file: NotationMathPrelims.tex
\longversion{
\section{Preliminaries and Notation}
\label{secMath}
}
As in~\cite{johnstone2018projective}, will work with a slight restriction of
problem~\eqref{ProbMono}, namely
\begin{eqnarray}\label{probCompMonoMany}
\longversion{0\in \sum_{i=1}^{n-1} G_i^* T_i(G_i z) +T_n(z).}
\shortversion{\textstyle{0\in \sum_{i=1}^{n-1} G_i^* T_i(G_i z) +T_n(z).}}
\end{eqnarray}
In terms of problem~\eqref{ProbMono}, we are simply requiring that $G_n$ be
the identity operator and thus that $\calH_n = \calH_0$. This is not much of a
restriction in practice, since one could redefine the last operator as 
$T_{n}\leftarrow G^*_{n} \circ T_{n} \circ G_{n}$, or one could
simply append a new operator $T_n$ with $T_n(z) = \{0\}$ everywhere. 

\label{secNot}
We will use a boldface $\bw = (w_1,\ldots,w_{n-1})$ for elements 
of $\calH_1\times\ldots\times\calH_{n-1}$.
To ease the presentation, we use the following notation
throughout, where $I$ denotes the identity operator:
\begin{align} \label{defwnpm}
G_{n}:\mathcal{H}_{n}\to\mathcal{H}_{n} &\triangleq I
&
(\forall\,k\in\mathbb{N}) \quad w_{n}^k &\triangleq -\textstyle{\sum_{i=1}^{n-1} G_i^*w_i^k}.
\end{align}
%
%
%
%
%
\longversion{
For any maximal monotone operator $A$ we will use the notation
\begin{eqnarray*}
	\prox_{\rho A} = (I+\rho A)^{-1}
\end{eqnarray*}
for any scalar $\rho > 0$ to denote the \emph{proximal operator}, 
also known as the backward or implicit step with respect to $A$.  This means that
\begin{eqnarray}\nonumber\label{defprox2}
x = \prox_{\rho A}(a) \quad\implies\quad \exists y\in Ax:x+\rho y = a.
\end{eqnarray}
The $x$ and $y$ satisfying this relation are unique. 
Furthermore, $\prox_{\rho A}$ is defined everywhere and
$\text{range}(\prox_A) = \text{dom}(A)$ \cite[Prop. 23.2]{bauschke2011convex}. 

By continuity, we mean in the strong topology defined in terms of the norm of
the given Hilbert space. That is, for all $g_0\in\calH_i$ and $\epsilon>0$,
there exists $\delta(g_0,\epsilon)$ s.t. whenever
$\|g_0-g\|\leq\delta(g_0,\epsilon)$, $\|T_i(g_0)-T_i(g)\|\leq\epsilon$.
Uniform continuity means that the constant is independent of $g_0$, i.e. the
above statement holds with $\delta(g_0,\epsilon)=\delta(\epsilon)$.

We use the standard ``$\rightharpoonup$" notation to denote weak convergence,
which is of course equivalent to ordinary convergence in finite dimensional
settings. 
}

%% file: AlgorithmOverview.tex
\section{Algorithm, Principal Assumptions, and Preliminary Analysis}
\label{secAlgOverview}
Algorithm~\ref{AlgfullyAsync} presents the algorithm analyzed in this paper. It is essentially
the block-iterative and potentially asynchronous projective splitting
algorithm as in~\cite{johnstone2018projective}, but directly incorporating a
backtracking linesearch procedure.

Let $\boldsymbol{\mathcal{H}} = \mathcal{H}_0\times
\mathcal{H}_1\times\cdots\times\mathcal{H}_{n-1}$ and $\calH_n = \calH_0$. The
algorithm produces a sequence of iterates denoted by $p^k =
(z^k,w_1^k,\ldots,w_{n-1}^k)\in\boldsymbol{\calH}$. Define the \emph{extended
solution set} or \emph{Kuhn-Tucker set} of~\eqref{probCompMonoMany} to be
\begin{eqnarray} \label{defCompExtSol}
\calS
	=  \left\{ (z,\bw) \in \boldsymbol{\mathcal{H}} \,\Big|\, 
	w_i\in T_i(G_i z),\, i=1,\ldots,n-1, 
	-\sum_{i=1}^{n-1} G_i^* w_i  \in T_n(z) \right\}.
\end{eqnarray}
Clearly $z\in\mathcal{H}_{0}$ solves~\eqref{probCompMonoMany} if and only if
there exists $\bw\in\calH_1 \times \cdots \times \calH_{n-1}$ such that
$(z,\bw)\in\calS$. 

Algorithm \ref{AlgfullyAsync} is a special case of a general seperator-projector method
for finding a point in a closed and convex set. At each iteration the method
constructs an affine function $\varphi_k:\calH^n\to \mathbb{R}$ which
separates the current point from the target set $\calS$ defined in
\eqref{defCompExtSol}. In other words, if $p^k$ is the current point in
$\boldsymbol{\calH}$ generated by the algorithm, $\varphi_k(p^k)>0$, and
$\varphi_k(p)\leq 0$ for all $p\in\calS$. The next point is then the
projection of $p^k$ onto the hyperplane $\{p:\varphi_k(p)=0\}$, subject to a
relaxation factor $\beta_k$. What makes projective splitting an operator
splitting method is that the hyperplane is constructed through individual
calculations on each operator $T_i$, either $\prox$ calculations or forward
steps. 

The hyperplane is defined in terms of the following affine function:
\begin{align}\label{defAffine}
\varphi_k(z,w_1,\ldots,w_{n-1}) &= \sum_{i=1}^{n-1}\inner{z-x^k_i}{y^k_i-w_i} 
+\Inner{z-x_i^n}{y_i^n + \sum_{i=1}^{n-1} w_i}.
\end{align} 
See \cite[Lemma 4]{johnstone2018projective} for the relevent properties of
$\varphi_k$. As in~\cite{johnstone2018projective}, we use the following inner
product and norm for $\boldsymbol{\calH}$, for an arbitrary scalar $\gamma > 0$:
\begin{align*} 
\Inner{(z^1,\bw^1)}{(z^2,\bw^2)}_\gamma &= 
\gamma\langle z^1,z^2\rangle + \textstyle{\sum_{i=1}^{n-1}\langle
w^1_i,w^2_i\rangle}
\\
\norm{(z,\bw)}_\gamma^2 &= \Inner{(z,\bw)}{(z,\bw)}_\gamma.
\end{align*}
Note that with this inner product it was shown in \cite[Lemma 4]{johnstone2018projective} that 
\begin{align}
\label{defGrad}
\nabla\varphi_k =
\left(\frac{1}{\gamma}\left(\sum_{i=1}^{n-1} G_i^* y_i^k+ y_n^k\right) \!\!,\;
x_1^k - G_1 x_{n}^k,x_2^k - G_2 x_{n}^k,\ldots,x_{n-1}^k - G_{n-1} x_{n}^k\right).
\end{align}
The scalar $\gamma>0$ controls the relative emphasis on the
primal and dual variables in the projection update in lines
\ref{eqAlgproj1}-\ref{eqAlgproj2}.

Algorithm \ref{AlgfullyAsync} has the following parameters:
\begin{itemize}
	\item For each iteration $k\geq 1$, a subset $I_k\subseteq \{1,\ldots,n\}$.
	\item For each $k\geq 1$ and $i=1,\ldots,n$, a positive scalar 
	stepsize $\rho_{i}^{k}$. For $i\in\Iforw$, $\rho_i^{k}$ is the initial stepsize tried in the backtracking linesearch while $\hat{\rho}_i^{k}$ is the accepted stepsize. 
	\item A constant $\nu\in(0,1)$ controlling how much the stepsize is decreased at each iteration of the backtracking linesearch. 
	\item For each iteration $k\geq 1$ and $i=1,\ldots,n$, a delayed 
	iteration index $d(i,k)\in\{1,\ldots,k\}$ which allows the subproblem calculations to use outdated information. 
	\item For each iteration $k\geq 1$, an overrelaxation parameter
	$\beta_k\in [\underline{\beta},\overline{\beta}]$ for some constants
	$0<\underline{\beta}\leq\overline{\beta}<2$.
	see \cite{johnstone2018projective} for more details. 
	\item Sequences of errors $\{e_i^k\}_{k\geq 1}$ for $i\in\Iback$, allowing
	us to model inexact computation of the proximal steps.
\end{itemize}

\longversion{
There are many ways in which Algorithm~\ref{AlgfullyAsync} could be
implemented in various parallel computing environments. We refer to
\cite{johnstone2018projective} for a more thorough discussion.
}

%% file: AsyncAlg-shortVersion.tex
\begin{algorithm}
\DontPrintSemicolon
\longversion{\algsetup{linenosize=\tiny} \footnotesize \setstretch{0.9}}
\SetKwInOut{Input}{Input}
\Input{$(z^1,{\bf w}^1)\in \boldsymbol{\mathcal{H}}$,
$(x_i^0,y_i^0)\in\mathcal{H}_i^2$ for $i=1,\ldots,n$,
$0<\underline{\beta}\leq\overline{\beta}<2$, $\gamma>0$, $\nu\in(0,1)$.}
\For{$k=1,2,\ldots$}
{   
	\For{$i=1,2,\ldots, n$}
	{
		\If{$i\in I_k$}
		{
			\If{$i\in\Iback$}
			{
			    	$a = G_i z^{d(i,k)}+\rho_{i}^{d(i,k)} w_i^{d(i,k)}+e_i^k$\label{lineaupdate}\;
			    $x_i^k = \prox_{\rho_{i}^{d(i,k)} T_i}(a)
			    $\label{LinebackwardUpdate}\;
			    $
			    y_i^k = (\rho_{i}^{d(i,k)})^{-1}
			    \left(
			    a - x_i^k
			    \right)
			    $\label{lineBackwardUpdateY}\;

		    }
		    \Else
		    {
				$\rho_i^{(1,k)} \leftarrow \rho_i^{d(i,k)}$\label{linebegintrack}\;
			    $\theta_i^k = G_{i}z^{d(i,k)}$ \label{lineTheta}\;
    			$\zeta_i^k = T_i \theta_i^k$   \label{lineZeta}  \;
    			\If{$\zeta_i^k = w_i^{d(i,k)}$}
    			{ 
    			$\hat{\rho}_i^{d(i,k)} \leftarrow \rho_i^{(j,k)},\;
    			x_i^k \leftarrow \theta_i^k,\;y_i^k \leftarrow \zeta_i^k$\label{lineQuickStop}\;
    		    }
    	        \Else 
    	        {
		     		\For{$j=1,2,\ldots$}{\label{lineStartFor}
			    		$\tilde{x}_i^{(j,k)} 
				    	= \theta_i^k - \rho_i^{(j,k)}(\zeta_i^k - w_i^{d(i,k)})$ \label{lineX}\;
	    				$\tilde{y}^{(j,k)}_i=T_i\tilde{x}^{(j,k)}_i$ \label{lineY}\;
	    				\If{\label{lineIf}$\Delta\|\theta_i^k-\tilde{x}_i^{(j,k)}\|^2 -
						\langle \theta_i^k-\tilde{x}_i^{(j,k)},\tilde{y}_i^{(j,k)}-w_i^{d(i,k)}\rangle
						\leq 0$}{
						 $
							\hat{\rho}_i^{d(i,k)} \leftarrow \rho_i^{(j,k)}, \;
							x_i^k \leftarrow \tilde{x}_i^{{(j,k)}}, \;
							y_i^k \leftarrow \tilde{y}_i^{{(j,k)}}$\;
							{\bf break}\;
						\label{lineBTreturn}
					}
					$\rho_i^{(j+1,k)} = \nu\rho_i^{(j,k)}$\label{lineendtrack}\;
			        }
		        $\hat{\rho}_i^{d(i,k)} \leftarrow \rho_i^{(j,k)}, x_i^k \leftarrow \tilde{x}_i^{(j,k)},y_i^k \leftarrow \tilde{y}_i^{(j,k)}$\label{lineendtracktrue}\;
			}

		    }
        }
	    \Else 
		{
			$(x_i^k,y_i^k)=(x_i^{k-1},y_i^{k-1})$
		}
	}
  $u_i^k = x_i^k - G_i x_n^k,\quad i=1,\ldots,n-1,$\label{lineCoordStart}\;
	$v^k = \sum_{i=1}^{n-1} G_i^* y_i^k+y_n^k$\label{lineVupdate}\;    
	$\pi_k = \|u^k\|^2+\gamma^{-1}\|v^k\|^2$  \label{linePiUpdate}\;    
	\eIf{$\pi_k>0$}{
		Choose some $\beta_k \in [\underline{\beta},\overline{\beta}]$\;  \label{lineHplane} 
		$\varphi_k(p^k) = 
		\langle z^k, v^k\rangle 
		+
		\sum_{i=1}^{n-1}
		\langle w_i^k,u_{i}^k\rangle 
		-
		\sum_{i=1}^{n}
		\langle x_i^k,y_i^k\rangle  
		$\label{lineComputeHplane}\;
		$\alpha_k = \frac{\beta_k}{\pi_k}\max\left\{0,\varphi_k(p^k)\right\}
		$\;
	}
	{
	 \If{$\cup_{j=1}^k I_j=\{1,\ldots,n\}$}
	 {
	 	\Return $z^{k+1}\leftarrow x_n^k, w_1^{k+1}\leftarrow y_1^k,\ldots,w_{n-1}^{k+1}\leftarrow y_{n-1}^k$\label{lineReturn}\;
	 }
	 \Else
	 {
	 	$\alpha_k = 0$\;
	 }	
	} 
$z^{k+1} = z^k - \gamma^{-1}\alpha_k v^k$\label{eqAlgproj1} \;
$w_i^{k+1} = w_i^k - \alpha_k u_{i}^k,\quad i=1,\ldots,n-1$,\label{eqAlgproj2} \;
$w_{n}^{k+1} = -\sum_{i=1}^{n-1} G_i^* w_{i}^{k+1}$\label{lineCoordEnd}\;
}
\caption{Asynchronous algorithm for solving~\eqref{probCompMonoMany}.}
\label{AlgfullyAsync}
\end{algorithm}

%% file: lemmas4weakconvGmod.tex
\newcommand{\bcH}{\boldsymbol{\mathcal{H}}}

\longversion{
Our main assumptions regarding~\eqref{probCompMonoMany} are
as follows:
}
\begin{assumption}
	\label{AssMonoProb}\label{assMono}	
	Problem~(\ref{probCompMonoMany}) conforms to the following:
	\begin{enumerate}
		 \item $\mathcal{H}_0 = \mathcal{H}_n$ and
		 $\mathcal{H}_1,\ldots,\mathcal{H}_{n-1}$ are real Hilbert spaces.
		\item For $i=1,\ldots,n$  the operators
		$T_i:\mathcal{H}_{i}\to2^{\mathcal{H}_{i}}$ are monotone. 
		\item For all $i$ in some subset $\Iforw \subseteq \{1,\ldots,n\}$, $\mathcal{H}_i$ is finite-dimensional,
		the operator $T_i$ is
		 continuous with respect to the metric induced by $\|\cdot\|$ (and thus single-valued),
		and $\text{dom}(T_i) = \mathcal{H}_i$.  
		\item For $i \in
		\Iback \triangleq \{1,\ldots,n\} \backslash \Iforw$, the operator 
		$T_i$ is maximal and the map $\prox_{\rho
			T_i}:\mathcal{H}_i\to\mathcal{H}_i$ 
 can be computed 
to within the error tolerance specified below in Assumption \ref{assErr}. 
		\item Each $G_i:\calH_{0}\to\calH_i$ for $i=1,\ldots,n-1$ is
		linear and bounded. 
		\item The solution set $\calS$ defined in
		(\ref{defCompExtSol}) is nonempty.
	\end{enumerate}
\end{assumption}


Our assumptions regarding the parameters of Algorithm \ref{AlgfullyAsync} are
as follows, and are the same as used
in~\cite{combettes2016async,eckstein2017simplified,johnstone2018projective}. 
\begin{assumption}\label{assAsync}
	For Algorithm \ref{AlgfullyAsync}, assume:
	\begin{enumerate}
		\item For some fixed integer $M\geq 1$, each index $i$ in $1,\ldots,n$
		is in $I_k$ at least once every $M$ iterations, that is, 
		$
		\bigcup_{k=j}^{j+M-1}I_k
		    = \{1,\ldots,n\}
		$
		for all $i=1,\ldots,n$ and $j \geq 1$.
		\item For some fixed integer $D\geq 0$, we have $k-d(i,k)\leq D$ for
		all $i, k$ with $i\in I_k$. 
	\end{enumerate}
\end{assumption}
We also use the following additional notation
from~\cite{eckstein2017simplified}: for all $i$ and $k$, define
\begin{align*}
\nonumber
S(i,k)
&=
\{j\in\mathbb{N}:j\leq k,i\in I_j \}
&
s(i,k)
&=
\left\{
\begin{array}{ll}
\max S(i,k),\quad& \text{when } S(i,k)\neq \emptyset\\
0,&\text{otherwise.}
\end{array}
\right.
\end{align*}
Essentially, $s(i,k)$ is the most recent iteration up to and including $k$ in
which the index-$i$ information in the separator was updated.
Assumption \ref{assAsync}
ensures that $0\leq k-s(i,k)< M$.
For all $i=1,\ldots,n$ and iterations $k$, also define
$
l(i,k) = d(i,s(i,k)),
$
the iteration in which the algorithm generated the
information $z^{l(i,k)}$ and $w_{i}^{l(i,k)}$ used to compute the current
point $(x_i^k,y_i^k)$. Regarding initialization, we set $d(i,0) = 0$; note that
the initial points $(x_i^0,y_i^0)$ are arbitrary.
%
We formalize the use of $l(i,k)$ in the following Lemma 
from \cite{johnstone2018projective}:
\begin{lemma}
	\label{lemUpdates}
Suppose Assumption \ref{assAsync}(1) holds. For all iterations $k\geq  M$ if Algorithm \ref{AlgfullyAsync} has not already terminated, the
updates can be written as
\input{updates}
\end{lemma}
\begin{proof}
Follows directly from~\cite[Lemma 6]{johnstone2018projective}, in view of how
$\hat{\rho}_i^k$ is calculated in Algorithm~\ref{AlgfullyAsync}. \myqed
\end{proof} 
Since Algorithm~\ref{AlgfullyAsync} is a projection method, it satisfies the
following lemma, identical to~\cite[Lemmas 2 and 6]{johnstone2018projective}:
\begin{lemma} 
	\label{propFejer}
	Suppose assumptions \ref{assMono} and \ref{assAsync}(1) hold. Then 
	for Algorithm \ref{AlgfullyAsync} 
\begin{enumerate}
		\item The sequence $\{p^k\}$ generated by Algorithm \ref{AlgfullyAsync} is bounded.
		\item\label{eqDif20} If Algorithm \ref{AlgfullyAsync} runs
		indefinitely, then $\|p^k - p^{k+1}\|\to 0$.
		\item\label{eqUpdateRewrite} Lines \ref{eqAlgproj1} and
		\ref{eqAlgproj2} may be written as
		\begin{align*}
		p^{k+1} 
		= p^k - \frac{\beta_k\max\{\varphi_k(p^k),0\}}
		{\|\nabla\varphi_k\|_\gamma^2}\nabla\varphi_k.
		\end{align*}
	\end{enumerate}
\end{lemma}
The assumptions regarding the proximal operator
evaluation errors are identical to those in~\cite{eckstein2017simplified}:
\begin{assumption}\label{assErr}
	The error sequences $\{\|e_i^k\|\}$ are bounded for all $i\in\Iback$. For some $\sigma$
	with $0\leq\sigma<1$ the following hold for all
	$k \geq 1$ such that Algorithm \ref{AlgfullyAsync} has not yet terminated:
\begin{align}
(\forall i\in\Iback)&&
	\langle G_i z^{l(i,k)} - x_i^k,e_i^{s(i,k)}\rangle
	&\geq
	-\sigma\|G_i z^{l(i,k)} - x_i^k\|^2
	\label{err1}\\
(\forall i\in \Iback)&&
	\langle e_i^{s(i,k)},y_i^k - w_i^{l(i,k)}\rangle 
	&\leq \rho^{l(i,k)}_i\sigma \|y_i^k - w_i^{l(i,k)}\|^2.
	\label{err2}
\end{align}
\end{assumption}
The stepsize assumptions differ from
\cite{johnstone2018projective,johnstone2018convergence} for $i\in\Iforw$ in
that we no longer assume Lipschitz continuity nor that the stepsizes are
bounded by the inverse of the Lipschitz constant. However, the initial trial
stepsize for the backtracking linesearch at each iteration is assumed to be
bounded from above and below:
\begin{assumption}\label{assStep}
In Algorithm~\ref{AlgfullyAsync},
\begin{align}\nonumber
\underline{\rho} &\triangleq 
    \min_{i=1,\ldots,n} \left\{ \inf_{k\geq 1} \rho_i^k\right\} > 0 &
\overline{\rho} &\triangleq
    \max_{i=1,\ldots,n} \left\{\sup_{k\geq 1} \rho_i^k \right\} < \infty.
\end{align}
\end{assumption}

%% file: updates.tex
\begin{align}
(\forall i\in \Iback) &&&
x^k_i+\rho_i^{l(i,k)}y^k_i 
=
G_i z^{l(i,k)}+\rho^{l(i,k)}_iw_i^{l(i,k)}+e_i^{s(i,k)}
&
y^k_i  &\in T_i x^k_i
\label{eqBack}	
\\
(\forall i\in\Iforw) &&&
x^k_{i} 
=
G_{i} z^{l(i,k)} - \hat{\rho}^{l(i,k)}_{i}\big(T_{i} G_{i} z^{l(i,k)} - w^{l(i,k)}_{i}\big)
&
y^k_{i} 
&=
T_i x^k_{i}.
\label{eqXForw}
\end{align}

%% file: proofMain.tex
\section{Weak Convergence to a Solution}\label{secMain}
\begin{lemma}
Suppose assumptions \ref{assMono}--\ref{assStep} hold. Then for all $k\in\rN$
and $i\in I_k$ such that Algorithm \ref{AlgfullyAsync} has not yet terminated,
the backtracking linesearch on lines \ref{linebegintrack}--\ref{lineendtracktrue}
terminates in a finite number of iterations. \label{lemFinite}
\end{lemma} 
\begin{proof}
For $k\geq M$, consider any $i\in\Iforw \cap I_k$ and assume that $T_i G_i
z^{l(i,k)}\neq w_i^{l(i,k)}$, since backtracking otherwise terminates
immediately at line \ref{lineQuickStop}. Using the definitions of $s(i,k)$
and $l(i,k)$ and some algebraic manipulation, the condition for terminating
the backtracking linesearch given on line~\ref{lineIf} may be written as:
\begin{align}
\frac{
	\langle G_i z^{l(i,k)} - \tilde{x}_i^{(j,s(i,k))},\tilde{y}_i^{(j,s(i,k))} - w_i^{l(i,k)}\rangle 
}{\|G_i z^{l(i,k)} - \tilde{x}_i^{(j,s(i,k))}\|^2}
\geq
\Delta.\label{eqLineScond}
\end{align}
For brevity, let  $\rho = \rho_i^{(j,s(i,k))} > 0$. Using lines
\ref{lineTheta}, \ref{lineZeta}, \ref{lineX}, and \ref{lineY}, the left-hand
side of \eqref{eqLineScond} may be written
\begin{align}
&\frac{
	\Big\langle
	T_iG_i z^{l(i,k)} - w_i^{l(i,k)},
	T_i\big(G_i z^{l(i,k)} - \rho(T_iG_i z^{l(i,k)} - w_i^{l(i,k)})\big) - w_i^{l(i,k)} 
	\Big\rangle
}{\rho\|T_iG_i z^{l(i,k)} - w_i^{l(i,k)}\|^2}.\label{eqFrac}
\end{align}
The numerator of this fraction may be expressed as
\begin{multline*}
\Big\langle
T_iG_i z^{l(i,k)} - w_i^{l(i,k)},
T_i\big(G_i z^{l(i,k)} - \rho(T_iG_i z^{l(i,k)} - w_i^{l(i,k)})\big) 
- 
T_i G_i z^{l(i,k)}
\Big\rangle
\\
+\|T_iG_i z^{l(i,k)} - w_i^{l(i,k)}\|^2.
\end{multline*}
Substituting this expression into \eqref{eqFrac} and applying the
Cauchy-Schwarz inequality to the inner product yields that the left-hand size
of \eqref{eqLineScond} is lower bounded by
\begin{align}
\label{eq2pass}
\frac{ 
	\|T_iG_i z^{l(i,k)} - w_i^{l(i,k)}\|
	- 
	\big\|T_i(G_i z^{l(i,k)} 
	    - \rho\big(T_iG_i z^{l(i,k)} - w_i^{l(i,k)})\big) - T_iG_i z^{l(i,k)}\big\|
}
{\rho\|T_iG_i z^{l(i,k)}-w_i^{l(i,k)}\|.}
\end{align}
The continuity of $T_i$ implies that the above expression tends to $+\infty$
as $\rho\to0$.  Since $\rho_j^{(j,k)}$ decreases geometrically to $0$ with $j$ on
line~\ref{lineendtrack}, it follows that~\eqref{eqLineScond} must eventually
hold.
\myqed
\end{proof}

\begin{lemma}
Suppose assumptions \ref{assMono}--\ref{assStep} hold, Algorithm
\ref{AlgfullyAsync} produces an infinite sequence of iterates, and both
\begin{enumerate}
	\item $G_i z^{l(i,k)} - x_i^k\to 0$ for all $i=1,\ldots,n$\label{lemEas}
	\item $y_i^k-w_i^{l(i,k)}\to 0$ for all $i=1\ldots,n$.\label{lempointhard}
\end{enumerate}
Then the
sequence $\{(z^k,\bw^k)\}$ generated by Algorithm~\ref{AlgfullyAsync}
converges weakly to some point $(\bar z,\overline{\bw})$ in the extended
solution set $\calS$ of~\eqref{probCompMonoMany} defined
in~\eqref{defCompExtSol}. Furthermore, $x_i^k\rightharpoonup G_i \bar z$
and $y_i^k\rightharpoonup \overline{w}_i$ for all $i=1,\ldots,n-1$,
$x_{n}^k\rightharpoonup \bar z$, and $y_n^k\rightharpoonup
-\sum_{i=1}^{n-1}G_i^* \overline{w}_i$.
\label{lemSumOld}
\end{lemma} 
\begin{proof}
First, note that 
$w_i^{l(i,k)}-w_i^k\to 0$ for all $i=1,\ldots,n$ and $z^{l(i,k)}-z^k \to
0$ \cite[Lemma 9]{johnstone2018projective}.
Combining $z^k - z^{l(i,k)}\to 0$ with point \eqref{lemEas} and the fact that $G_i$ is bounded, we obtain  that
$
G_iz^k - x_i^k\to 0
$
for $i=1,\ldots,n$. Similarly, combining $w_i^{l(i,k)}-w_i^k\to 0$ with point \eqref{lempointhard} we have 
$
y_i^k-w_i^{k}\to 0.
$
The proof is now identical to part 3 of the 
proof of \cite[Theorem~1]{johnstone2018projective}. \myqed
\end{proof} 
\longversion{Before commencing with the final two technical lemmas, we need two definitions.}
Define $\phi_k = \varphi_k(p^k)$ and 
\begin{align}
\label{defPsi}
(\forall i=1,\ldots,n) \quad
\psi_{ik} &\triangleq \langle G_i z^{l(i,k)} - x^k_i,y_i^k - w^{l(i,k)}_i\rangle
&
\psi_k &\triangleq \textstyle{\sum_{i=1}^{n}\psi_{ik}}.
\end{align}

\begin{lemma}\label{lemBound}
	Suppose assumptions \ref{assMono}--\ref{assStep} hold and that Algorithm
	\ref{AlgfullyAsync} produces an infinite sequence of iterates with
	$\{x_i^k\}$ and $\{y_i^k\}$ being bounded. Then, for all $i=1,\ldots,n$,
	it holds that $G_i z^{l(i,k)} - x_i^k\to 0$.
\end{lemma}
\begin{proof}



Using
\eqref{defGrad}
\begin{eqnarray}
\|\nabla \varphi_k\|_\gamma^2
=
\gamma^{-1}\left\|\textstyle{\sum_{i=1}^{n-1} G_i^* y^k_i}+y_{n}^k\right\|^2
+
\textstyle{\sum_{i=1}^{n-1}
\|x^k_i - G_i x_{n}^k\|^2}
\label{eqGrad}.
\end{eqnarray}
By assumption, $\{x_i^k\}$ and $\{y_i^k\}$ are bounded sequences, 
therefore $\{\|\nabla \varphi_k\|_\gamma\}$ is bounded; 
let $\xi_1 > 0$ be some bound on this sequence. 
Next, we will establish that there exists some $\xi_2>0$ such that
\begin{equation}\label{eqlem12}
\psi_k
\geq
\xi_2
\textstyle{\sum_{i=1}^n
\| G_i z^{l(i,k)}-x^k_i\|^2}
.
\end{equation}
The proof resembles that of~\cite[Lemma 12]{johnstone2018projective}: 
since the
backtracking linesearch terminates in a finite number of iterations, we must have
\begin{align}\label{eqLB1}
\langle G_i z^{l(i,k)} - x_i^k,y_i^k - w_i^{l(i,k)}\rangle 
\geq 
\Delta
\|G_i z^{l(i,k)} - x_i^k\|^2
\end{align}
for every $k\in\rN$ and $i\in\Iforw$. Terms in $\Iback$ are treated as before
in \cite[Lemma 12]{johnstone2018projective}: 
specifically, for all $i\in\Iback$,
\begin{align}
	\psi_{ik}
&=
\left\langle G_i z^{l(i,k)}-x^k_i,y_i^k-w^{l(i,k)}_i\right\rangle
\nonumber\\
&\overset{\text{(a)}}{=}
\left\langle G_i z^{l(i,k)}-x^k_i,
\big(\rho_i^{l(i,k)}\big)^{-1}\left(G_i z^{l(i,k)}-x^k_i+e_i^{s(i,k)}\right)
\right\rangle
\nonumber\\
&=
\big(\rho^{l(i,k)}_i\big)^{-1}\|G_i z^{l(i,k)}-x^k_i\|^2 
+\big(\rho^{l(i,k)}_i\big)^{-1}
\left\langle G_i z^{l(i,k)}-x^k_i,e_i^{s(i,k)}\right\rangle
\nonumber\\
&\overset{\text{(b)}}{\geq}
(1-\sigma)\big(\rho^{l(i,k)}_i\big)^{-1}\|G_i z^{l(i,k)}-x^k_i\|^2
\label{eqForLineS}.
\end{align}
In the above derivation, (a) follows by substitution of~\eqref{eqBack} and
(b) is justified by using~\eqref{err1}
in Assumption~\ref{assErr}.
Combining \eqref{eqLB1} and \eqref{eqForLineS} yields
\begin{align}
\psi_k\geq 
(1-\sigma)\overline{\rho}^{-1}
\sum_{i\in\Iback}
\|G_i z^{l(i,k)}-x^k_i\|^2
+
\Delta\sum_{i\in\Iforw}\|G_i z^{l(i,k)}-x^k_i\|^2,
\end{align}
which yields \eqref{eqlem12} with $\xi_2 = \min\{(1-\sigma)\overline{\rho}^{-1},\Delta\}>0$. 



We now proceed as in
as in part 1 of the proof of
\cite[Theorem~1]{johnstone2018projective}:
first,
Lemma \ref{propFejer}(\ref{eqUpdateRewrite}) states that 
the updates on lines \ref{eqAlgproj1}--\ref{eqAlgproj2}
can be written as
\[
p^{k+1} 
= p^k - \frac{\beta_k\max\{\phi_k,0\}}
{\|\nabla\varphi_k\|_\gamma^2}\nabla\varphi_k. 
\]
Lemma~\ref{propFejer}(\ref{eqDif20}) guarantees that $p^k - p^{k+1} \to 0$, so it follows that
\[
0 = \lim_{k\to\infty}\|p^{k+1}-p^k\|_\gamma 
= \lim_{k\to\infty}\frac{\beta_k\max\{\phi_k,0\}}{\|\nabla\varphi_k\|_\gamma}
\geq\frac{\underline{\beta}\limsup_{k\to\infty}\max\{\phi_k,0\}}{\sqrt{\xi_1}}.
\]
Therefore, $\limsup_{k\to\infty} \phi_k\leq 0$. Since \cite[Lemma
10]{johnstone2018projective} states that  $\phi_k -
\psi_k\to 0$, it follows that $\limsup_{k\to\infty}\psi_k\leq 0$. With (a)
following from \eqref{eqlem12}, we next obtain
\begin{align*}
0
\geq \limsup_{k\to\infty} \psi_k
&\overset{\text{(a)}}{\geq}
\xi_2\limsup_k \textstyle{\sum_{i=1}^{n}\|G_i z^{l(i,k)} - x^k_i\|^2}
\\
&\geq 
\xi_2\liminf_k \textstyle{\sum_{i=1}^{n}\|G_i z^{l(i,k)} - x^k_i\|^2}
\geq 0.
\end{align*} 
Therefore,
$
G_i z^{l(i,k)}-x_i^k\to 0
$
for $i=1,\ldots,n$. 
\myqed
\end{proof}

\begin{lemma}\label{lemBound2}
	Suppose assumptions \ref{assMono}--\ref{assStep} hold and that Algorithm
	\ref{AlgfullyAsync} produces an infinite sequence of iterates with
	$\{x_i^k\}$ and $\{y_i^k\}$ being bounded. Then, for all $i\in\Iback$, one
	has  $y_i^k - w_i^{l(i,k)}\to 0$.
\end{lemma}
\begin{proof}
The argument to is siimlar to those of
\cite[Lemma~13]{johnstone2018projective} and \cite[Theorem~1
(part~2)]{johnstone2018projective}:  the crux of the proof is 
to establish for all $k\geq M$ that
\begin{align}
\psi_k 
+ \sum_{i\in\Iforw}\langle x^k_{i}-G_{i}z^{l(i,k)},T_i x^k_{i}-T_i G_{i}z^{l(i,k)}\rangle
\geq
(1-\sigma)\underline{\rho}
\sum_{i\in\Iback} \|y_i^k - w_i^{l(i,k)}\|^2.
  \label{eqPhiLB2}
\end{align}
Since $T_i$ is continuous and defined everywhere, $x_i^k$ is bounded by
assumption, and $z^{l(i,k)}$ is bounded by Lemma \ref{propFejer}, the extreme
value theorem implies that $T_i x^k_{i}-T_i G_{i}z^{l(i,k)}$ is bounded.
Furthermore from Lemma~\ref{lemBound},
$
\lim\sup_{k\to\infty}\psi_k\leq 0,
$
and $x_i^k - G_iz^{l(i,k)}\to 0$.
Therefore the desired result follows from~\eqref{eqPhiLB2}.

It remains to prove \eqref{eqPhiLB2}. 
For all $k\geq  M$, we have 
\allowdisplaybreaks
\begin{align}
\nonumber
\psi_k &= \sum_{i=1}^n\langle G_i z^{l(i,k)}-x^k_i,y^k_i-w^{l(i,k)}_i\rangle \\
&\overset{\text{(a)}}{=}
\sum_{i\in\Iback}
\langle \rho^{l(i,k)}_i(y^k_i - w^{l(i,k)}_i) - e_i^{s(i,k)},y^k_i - w^{l(i,k)}_i\rangle
\nonumber\\&\qquad
+
\sum_{i\in\Iforw}\langle G_{i} z^{l(i,k)}-x^k_{i},T_i G_{i} z^{l(i,k)}-w_{i}^{l(i,k)}\rangle
\nonumber\\&\qquad
+
\sum_{i\in\Iforw}
\langle G_{i}z^{l(i,k)}-x^k_{i},y^k_{i}-T_i G_{i} z^{l(i,k)}\rangle
\nonumber\\\nonumber&\overset{\text{(b)}}{=}
\sum_{i\in\Iback}
\left(
\rho^{l(i,k)}_i\|y^k_i - w^{l(i,k)}_i\|^2 - \langle e_i^{s(i,k)},y^k_i - w^{l(i,k)}_i\rangle
\right)
\\\label{dropForward}
&\qquad
+
\sum_{i\in\Iforw}\langle \rho^{l(i,k)}_{i}(T_iG_{i} z^{l(i,k)}-w_{i}^{l(i,k)}),T_i G_{i} z^{l(i,k)}-w_{i}^{l(i,k)}\rangle
\\\nonumber
&\qquad
-
\sum_{i\in\Iforw}\langle x^k_{i}-G_{i}z^{l(i,k)},T_i x^k_{i}-T_i G_{i}z^{l(i,k)}\rangle
\\
&\overset{\text{(c)}}{\geq}
(1-\sigma)\underline{\rho}\sum_{i\in\Iback}\|y^k_i - w^{l(i,k)}_i\|^2
\nonumber
-
\sum_{i\in\Iforw}\langle x^k_{i}-G_{i}z^{l(i,k)},T_i x^k_{i}-T_i G_{i}z^{l(i,k)}\rangle.
\end{align}
In the above derivation, (a) follows by substition of~\eqref{eqBack} into the
$\Iback$ terms and algebraic manipulation of the $\Iforw$ terms.  Next, (b) is
obtained by algebraic simplification of the $\Iback$ terms and substitution
of~\eqref{eqXForw} into the two groups of $\Iforw$ terms. Finally, (c) follows
by substituting the error criterion~\eqref{err2} from Assumption~\ref{assErr}
into the $\Iback$ terms and dropping the terms
from~\eqref{dropForward}, which must be nonnegative.  \myqed
\end{proof}

\begin{theorem}\label{thmjustcont}
Suppose assumptions \ref{assMono}--\ref{assStep} hold. If
    Algorithm~\ref{AlgfullyAsync} terminates at line \ref{lineReturn}, then
    its final iterate $(z^{k+1},\bw^{k+1})$ is a member of the extended
    solution set $\calS$ defined in~\eqref{defCompExtSol}. Otherwise, the
    sequence $\{(z^k,\bw^k)\}$ generated by Algorithm~\ref{AlgfullyAsync}
    converges weakly to some point $(\bar z,\overline{\bw})$ in $\calS$ and
    furthermore $x_i^k\rightharpoonup G_i \bar z$ and $y_i^k\rightharpoonup
    \overline{w}_i$ for all $i=1,\ldots,n-1$, $x_{n}^k\rightharpoonup \bar z$,
    and $y_n^k\rightharpoonup -\sum_{i=1}^{n-1}G_i^* \overline{w}_i$.
\end{theorem}
\begin{proof}
The argument when the algorithm terminates via line \ref{lineReturn} is
identical to \cite[Theorem 1]{johnstone2018projective}.
From now on we assume the algorithm produces an infinite sequence of
iterates. The proof proceeds by showing that the two conditions of Lemma
\ref{lemSumOld} are satisfied.  To establish Lemma~\ref{lemSumOld}(\ref{lemEas}) and Lemma~\ref{lemSumOld}(2) for $i\in\Iback$,
we will show that $\{x_i^k\}$ and $\{y_i^k\}$ are bounded, and then employ
Lemmas~\ref{lemBound} and \ref{lemBound2}. This argument is only a slight variation of what was given in \cite{johnstone2018projective}.
The main departure from~\cite{johnstone2018projective} is in 
establishing Lemma~\ref{lemSumOld}(\ref{lempointhard}) for $i\in\Iforw$, which requires significant innovation.



We begin by establishing that $\{x_i^k\}$ and $\{y_i^k\}$ are bounded. 
For $i\in\Iback$ the boundedness of $\{x_i^k\}$ follows exactly the same argument as \cite[Lemma 10]{eckstein2017simplified}. For $i\in\Iforw$ write using Lemma \ref{lemUpdates}
\begin{align}
\|x_i^k\|
&\leq
\|G_i z^{l(i,k)} - \hat{\rho}_i^{l(i,k)}T_i G_i z^{l(i,k)}\|
+
\hat{\rho}_i^{l(i,k)}\|w_i^{l(i,k)}\|
\label{eqxbound1}\\\label{eqxbound2}
&\leq 
\|G_i\| \|z^{l(i,k)}\| + \overline{\rho}\|T_i G_i z^{l(i,k)}\|
+
\overline{\rho}\|w_i^{l(i,k)}\|.
\end{align}
Now $z^{l(i,k)}$ and $w_i^{l(i,k)}$ are bounded by Lemma \ref{propFejer}.
Furthermore, since $T_i$ is continuous with full domain, $G_i$ is bounded, and
$z^{l(i,k)}$ is bounded, $\{T_iG_i z^{l(i,k)}\}$ is bounded by the extreme
value theorem.  Thus $\{x_i^k\}$ is bounded for $i\in\Iforw$. 

Now we prove that $\{y_i^k\}$ is bounded. For $i\in\Iback$, Lemma
\ref{lemUpdates} implies that
\begin{align*}
y_i^k = \left(\rho_i^{l(i,k)}\right)^{-1}
\left(
G_i z^{l(i,k)} - x_i^k + \rho_i^{l(i,k)} w_i^{l(i,k)}
+
e_i^{s(i,k)}
\right).
\end{align*}
Since $\rho_i^k$ is bounded from above and below, $G_i$ is bounded, $z^{l(i,k)}$ and
$w_i^{l(i,k)}$ are bounded by Lemma \ref{propFejer}, and $e_i^{s(i,k)}$ is assumed to be
bounded, $\{y_i^k\}$ is bounded for $i\in\Iback$. For $i\in\Iforw$, since
$y_i^k = T_i x_i^k$ and $T_i$ is continuous with full domain, it follows again
from the extreme value theorem that $\{y_i^k\}$ is bounded.

Therefore we can apply Lemma~\ref{lemBound} to infer that $G_i z^{l(i,k)} -
x_i^k\to 0$ for $i=1,\ldots,n$, and Lemma~\ref{lemSumOld}(\ref{lemEas}) holds.
Furthermore we can apply Lemma~\ref{lemBound2} to infer that $y_i^k -
w_i^{l(i,k)}\to 0$ for $i\in\Iback$.

It remains to establish that $y_i^k - w_i^{l(i,k)}\to 0$ for $i\in\Iforw$. The
argument needs to be significantly expanded from that in
\cite{johnstone2018projective}, since it is not immediate that the stepsize
$\hat{\rho}_i^k$ is bounded away from $0$. 

From Lemma \ref{propFejer}, we know that
$z^{l(i,k)}$ and $w_i^{l(i,k)}$ are bounded, as is the operator $G_i$ by
assumption. Furthermore, since $T_i$ is continuous with full domain, we know
once again from the extreme value theorem that there exists $B\geq 0$ such that
\begin{align}\label{eqBoundB}
(\forall k\in\mathbb{N}) \quad
\|T_i G_i z^{l(i,k)} - w_i^{l(i,k)}\|\leq B.
\end{align} 
We have already shown that $x_i^k$ is bounded.  Using the boundedness of $z^k$
and $w_i^k$ in conjunction with Assumption~\ref{assStep} and inspecting the
steps in the backtracking search, there must exist a closed ball
$\mathcal{B}_x\subset\mathcal{H}_i$ such that $\tilde{x}_i^{(j,s(i,k))}\in\mathcal{B}_x$ for all
$k,j\in\mathbb{N}$ such that $i \in I_k$ and $j$ is encountered during the
backtracking linesearch at step $k$.
In addition, let
$\mathcal{B}_{GZ}\subset\mathcal{H}_i$ be a closed ball containing $G_iz^{l(i,k)}$
for all $k\in\mathbb{N}$. Let $\mathcal{B} =
\mathcal{B}_x\cup\mathcal{B}_{GZ}$, which is another closed ball. Since $\mathcal{H}_i$ is finite dimensional, $\mathcal{B}$ is compact.
Since $T_i$ is continous everywhere, by the Heine-Cantor theorem it is
uniformly continuous on $\mathcal{B}$ \cite[Theorem
21.4]{ross1980elementary}.

%% file: proofWithUniform.tex

Continuing, we write
\begin{align}\label{eq2y}
y_i^k -w_i^{l(i,k)}= T_ix_i^k-w_i^{l(i,k)} 
=  T_i G_i z^{l(i,k)}-w_i^{l(i,k)} + T_i x_i^k - T_i G_i z^{l(i,k)}.
\end{align}
Since $T_i$ is uniformly continuous on $\mathcal{B}$ it must be Cauchy
continuous, meaning that $x_i^k - G_i z^{l(i,k)}\to0$ implies $T_i
x_i^k - T_i G_i z^{l(i,k)}\to 0$. Thus, to prove that $y_i^k -
w_i^{l(i,k)}\to 0$ it is sufficient to show that $T_i G_i
z^{l(i,k)}-w_i^{l(i,k)}\to 0$.

We now show that indeed $T_i G_i
z^{l(i,k)}-w_i^{l(i,k)}\to 0$. 
Fix $\epsilon>0$. Since $T_i$ is uniformly continuous on $\mathcal{B}$,
there exists $\delta>0$ such that  whenever
$x,y \in \mathcal{B}$ and $\|x-y\|\leq\delta$, then
$\|T_ix-T_iy\|\leq\epsilon/4$.
Since $G_i z^{l(i,k)} - x_i^k\to 0$, there exists $K\geq 1$ 
such that for all $k\geq K$,
\begin{align}\label{eqLimit}
\|G_i z^{l(i,k)} - x_i^k\|\leq 
\epsilon 
\min\left(
\frac{\nu\epsilon}{4B\Delta},\frac{\nu\delta}{B},\underline{\rho} 
\right)
\end{align}
with $B$ as in \eqref{eqBoundB}, $\Delta$ from the linesearch termination
criterion, and $\underline{\rho}$ from Assumption \ref{assStep}. For any
$k\geq K$ we will show that
\begin{align}
\label{eqWhatWant}
\|T_i G_i z^{l(i,k)} - w_i^{l(i,k)}\|\leq \epsilon.
\end{align}
If 
$
\|T_i G_i z^{l(i,k)} - w_i^{l(i,k)}\| \leq \epsilon/2,
$
then \eqref{eqWhatWant} clearly holds.  So from now on it is sufficient to
consider $k$ for which
$
\|T_i G_i z^{l(i,k)} - w_i^{l(i,k)}\|> \epsilon/2.
$
Again, let $\rho_i^{(j,s(i,k))} = \rho$ for brevity. Reconsidering \eqref{eq2pass}, we now
have the following lower bound for the left-hand side of \eqref{eqLineScond}:
\begin{equation*}
\frac{ 
	\|T_iG_i z^{l(i,k)} - w_i^{l(i,k)}\|
	-
	\Big\|T_i\big(G_i z^{l(i,k)} - \rho(T_iG_i z^{l(i,k)} - w_i^{l(i,k)})\big) 
	         - T_iG_i z^{l(i,k)}\Big\|
}
{\rho\|T_iG_i z^{l(i,k)}-w_i^{l(i,k)}\|}
\end{equation*}
\vspace{-1.6ex}
\begin{equation}
>
\frac{ 
	\epsilon/2
	-
	\Big\|T_i\big(G_i z^{l(i,k)} - \rho(T_iG_i z^{l(i,k)} - w_i^{l(i,k)})\big) 
	         - T_iG_i z^{l(i,k)}\Big\|
}
{\rho \|T_iG_i z^{l(i,k)}-w_i^{l(i,k)}\|}.
\label{weirdsqueeze}
\end{equation}
Now,  suppose it were true that
\begin{align}\label{eqStep}
\|G_i z^{l(i,k)}-\tilde{x}_i^{(j,s(i,k))}\| 
   = \rho\|T_i G_i z^{l(i,k)} - w_i^{l(i,k)}\|\leq \delta.
\end{align}
Then the uniform continuity of $T_i$ on $\mathcal{B}$ would imply that
\begin{multline*}
\|T_i G_i z^{l(i,k)}-T_i \tilde{x}_i^{(j,s(i,k))}\|
\shortversion{\\}
=
\Big\|T_i\big(G_i z^{l(i,k)} 
        - \rho(T_iG_i z^{l(i,k)} - w_i^{l(i,k)})\big) - T_iG_i z^{l(i,k)}\Big\|
\leq 
\frac{\epsilon}{4}.
\end{multline*}
We next observe that \eqref{eqStep} is implied by
$
\rho\leq\frac{\delta}{B},
$
in which case~\eqref{weirdsqueeze} gives the following lower bound on the left-hand side of~\eqref{eqLineScond}:
\begin{align*}
\frac{
	\langle G_i z^{l(i,k)} - \tilde{x}_i^{(j,s(i,k))},\tilde{y}_i^{(j,s(i,k))} - w_i^{l(i,k)}\rangle 
}{\|G_i z^{l(i,k)} - \tilde{x}_i^{(j,s(i,k))}\|^2}
>
\frac{\epsilon}{4\rho\|T_i G_i z^{l(i,k)} - w_i^{l(i,k)}\|}
\geq
\frac{\epsilon}{4\rho B}.
\end{align*}
Therefore if $\rho$ also satisfies
$
\rho\leq \frac{\epsilon}{4B\Delta}
$
, then 
\begin{align}
\frac{
	\langle G_i z^{l(i,k)} - \tilde{x}_i^{(j,s(i,k))},\tilde{y}_i^{(j,s(i,k))} - w_i^{l(i,k)}\rangle 
}{\|G_i z^{l(i,k)} - \tilde{x}_i^{(j,s(i,k))}\|^2}
>
\Delta. \label{eqTerm}
\end{align}
Thus, any stepsize satisfying
$\rho \leq (1/B) \min\left\{ \epsilon/4\Delta, \delta \right\}$
must cause the backtracking linesearch termination criterion at line
\ref{lineIf} to hold. Therefore, since the backtracking linesearch proceeds by
reducing the stepsize by a factor of $\nu$ at each inner iteration, it must terminate with
\begin{align}
\hat{\rho}_i^{l(i,k)}\geq 
\underline{\rho}^{bt}
\triangleq
\min\left\{
\frac{\nu\epsilon}{4 B\Delta}
,
\frac{\nu\delta}{B},
\underline{\rho}
\right\}.\label{eqstepLB}
\end{align}
Now, using Lemma \ref{lemUpdates}, we have 
\begin{align*}
x_i^{k} - G_i z^{l(i,k)} = -\hat{\rho}_i^{l(i,k)}(T_i G_i z^{l(i,k)} - w_i^{l(i,k)})
\end{align*}
and therefore
\begin{align*}
\|x_i^{k} - G_i z^{l(i,k)}\| 
   = \hat{\rho}_i^{l(i,k)}\|T_i G_i z^{l(i,k)} - w_i^{l(i,k)}\|.
\end{align*}
Thus,
\begin{align*}
\|T_i G_i z^{l(i,k)} - w_i^{l(i,k)}\|
&\leq
(\underline{\rho}^{bt})^{-1}
\|x_i^{k} - G_i z^{l(i,k)}\|
\\
&\leq 
\min\left\{
\frac{\nu\epsilon}{4 B\Delta}
,
\frac{\nu\delta}{B},
\underline{\rho}
\right\}^{-1}
\|x_i^{k} - G_i z^{l(i,k)}\| \leq \epsilon.
\end{align*}
and therefore \eqref{eqWhatWant} holds for all $k\geq K$. Since $\epsilon>0$
was chosen arbitrarily, it follows that $\|T_i G_i z^{l(i,k)} -
w_i^{l(i,k)}\|\to 0$ and thus $\|y_i^k - w_i^{l(i,k)}\|\to 0$ by \eqref{eq2y}.
The proof that Lemma~\ref{lemSumOld}(\ref{lempointhard}) holds is now complete.
The proof of the theorem now follows from Lemma~\ref{lemSumOld}. 
 \myqed
\end{proof}
If $\mathcal{H}_i$ is not finite dimensional for $i\in\Iforw$, Theorem
\ref{thmjustcont} can still be proved if the assumption on $T_i$ is
strengthened to Cauchy continuity over all bounded sequences. This is slightly
stronger than the assumption given in \cite[Equation (1.1)]{tseng2000modified}
for proving weak convergence of Tseng's forward-backward-forward method. That
assumption is Cauchy continuity but only for all \emph{weakly convergent}
sequences.